\documentclass[reqno]{amsart}
\usepackage[a4paper, margin=2cm] {geometry}
\usepackage{etoolbox}

\patchcmd{\section}{\scshape\centering}{\bfseries\Large}{}{}
\makeatletter
\renewcommand{\@secnumfont}{\bfseries\Large}
\makeatother

\makeatletter

\def\enddoc@text{\ifx\@empty\@translators \else\@settranslators\fi}

\g@addto@macro\@setauthors{
  \par\vspace{2ex} 
  \@setaddresses
}
\makeatother

\patchcmd{\subsection}{\bfseries}{\bfseries\large}{}{}
\patchcmd{\subsection}{-.5em}{.5ex plus .2ex}{}{}

\usepackage[utf8]{inputenc}
\usepackage{amsmath}
\usepackage{siunitx}
\usepackage{amsfonts}
\usepackage{graphicx}
\usepackage{wrapfig}
\usepackage{subfig}
\usepackage{listings}
\usepackage{mathabx}
\usepackage{hyperref}
\usepackage{color} 
\usepackage{amssymb,amscd}
\usepackage{enumerate}
\usepackage{etoolbox}
\usepackage{tikz}
\usepackage{faktor}
\usepackage{graphicx}
\usepackage[utf8]{inputenc}
\usepackage{amsmath}
\usepackage{siunitx}
\usepackage{amsfonts}
\usepackage{graphicx}
\usepackage{wrapfig}
\usepackage{subfig}
\usepackage{listings}
\usepackage{mathabx}
\usepackage{scalerel}
\usepackage{stackengine}
\usepackage{color}

\numberwithin{equation}{section}

\newtheorem{theorem}{Theorem}
\newtheorem{lemma}[theorem]{Lemma}

\newtheorem{proposition}[theorem]{Proposition}

\theoremstyle{definition}

\newtheorem{problem}{Problem}

\newtheorem{definition}[theorem]{Definition}

\usepackage{thmtools}
\usepackage{thm-restate}

\makeatletter
\newcommand{\restatecor}[2]{%
  \begingroup
  \let\c@theorem\c@restatecorcounter
  \def\thetheorem{#1}%
  \def\theHtheorem{restatecor.#1}%
  #2*%
  \endgroup
}
\makeatother

\usepackage{mathtools}
\usepackage{enumitem}

\usepackage{overpic} 
\usepackage{url} 
\usepackage{comment} 

\newcommand{\dd}{\displaystyle\int\limits_{\mathbb{R}^d}}

\newcommand{\rr}{\mathbb{R}}

\newcommand{\lip}{\mathrm{Lip}_0}
\newcommand{\cw}{\overset{w}{\to}}

\newcommand{\comp}{\overset{c}{\hookrightarrow}}
\newcommand{\FF}{\mathcal{F}}

\newcommand{\Lip}{\mathrm{Lip}}

\date{}
\usepackage{amsrefs}

\hypersetup{colorlinks=false}

\usepackage[nameinlink]{cleveref}

\crefformat{equation}{\begingroup\hypersetup{linkbordercolor=cyan}#2eq.~(#1)#3\endgroup}
\Crefformat{equation}{\begingroup\hypersetup{linkbordercolor=cyan}#2Eq.~(#1)#3\endgroup}

\crefformat{theorem}{\begingroup\hypersetup{linkbordercolor=red}#2theorem~#1#3\endgroup}
\Crefformat{theorem}{\begingroup\hypersetup{linkbordercolor=red}#2Theorem~#1#3\endgroup}

\crefformat{lemma}{\begingroup\hypersetup{linkbordercolor=red}#2lemma~#1#3\endgroup}
\Crefformat{lemma}{\begingroup\hypersetup{linkbordercolor=red}#2Lemma~#1#3\endgroup}

\crefformat{proposition}{\begingroup\hypersetup{linkbordercolor=red}#2proposition~#1#3\endgroup}
\Crefformat{proposition}{\begingroup\hypersetup{linkbordercolor=red}#2Proposition~#1#3\endgroup}

\crefformat{observation}{\begingroup\hypersetup{linkbordercolor=red}#2observation~#1#3\endgroup}
\Crefformat{observation}{\begingroup\hypersetup{linkbordercolor=red}#2Observation~#1#3\endgroup}

\newcommand{\eqcref}[1]{%
  \begingroup
  \hypersetup{linkbordercolor=cyan}%
  \eqref{#1}%
  \endgroup
}

\makeatletter

\pretocmd{\@setthanks}{\par\noindent}{}{}
\makeatother
\title{The Dunford--Pettis property for Lipschitz-free spaces over $\rr^n$}
\author{Fraser Mason}
\address{Department of Pure Mathematics and Mathematical Statistics, Centre for Mathematical Sciences, University of Cambridge, 
  Wilberforce Road, Cambridge CB3 0WB, United Kingdom}
\email{fm531@cam.ac.uk}

\begin{document}

\begin{abstract}
    \noindent We show that for every $n\in \mathbb{N}$, the Lipschitz-free space $\FF(\rr^n)$ has the Dunford--Pettis property. As a consequence, for every $n\in \mathbb{N}$ and every nonempty $M\subset \rr^n$, the space $\FF(M)$ has the Dunford--Pettis property. This implies that $\FF(M)$ is not isomorphic to $\FF(X)$ for any separable Banach space $X$ that fails the Dunford--Pettis property. In particular, this applies to $X=\ell_2$, and more generally to every infinite-dimensional separable reflexive space. We also deduce that if $1<p<\infty$, then $\FF(\ell_p)$ is not isomorphic to $\left(\bigoplus\limits_{n=1}^\infty \FF(\ell_p^n) \right)_{\ell_1}$.
\end{abstract}
\maketitle

\section{Introduction} 

\subsection{Main results} 
\noindent Given a metric space $M$ and a basepoint $0_M\in M$, $\lip(M)$ denotes the space of Lipschitz functions $f:M\to \mathbb{R}$ with $f(0_M)=0$, which is a Banach space with respect to the Lipschitz norm. For $x\in M$, the map $\delta_M(x):\lip(M)\to \mathbb{R}$ defined by $\delta_M(x)(f):=f(x)$ is an element of $\lip(M)^{\star}$, and the \emph{Lipschitz-free space} $\FF(M)$ is defined to be the subspace $\overline{\mathrm{span}}^{\lVert \cdot \rVert}\{\delta_M(x):x\in M\}\subset \lip(M)^{\star}$. The isometry classes of both $\lip(M)$ and $\FF(M)$ are independent of the choice of basepoint. \\

\noindent Our main results are as follows. To the best of our knowledge, the question of whether $\FF(\rr^d)$ has the Dunford--Pettis property was open for $d\geq 2$. In this paper, we give a positive answer to this question. 

\begin{restatable}{theorem}{thmone}
\label{dunfordpettis}
    Let $d\in \mathbb{N}$. Then $\FF(\rr^d)$ has the Dunford--Pettis property.
\end{restatable}

\noindent We also obtain three corollaries, all of which follow quickly from Theorem \ref{dunfordpettis}. The first states that the Lipschitz-free space over any nonempty subset of $\rr^d$ has the Dunford--Pettis property.

\begin{restatable}{corollary}{cortwo}\label{M}
    Let $d\in \mathbb{N}$ and $M$ be a nonempty subset of $\rr^d$. Then $\FF(M)$ has the Dunford--Pettis property.
\end{restatable}

\noindent \noindent We are not aware of any result in the literature resolving the question of whether $\FF(\rr^2)$ and $\FF(\ell_2)$ are isomorphic. In the introduction to \cite{pelczynski}, it was asked whether, for Banach spaces $E$ and $X$ with $\dim E<\infty$ and $\dim X=\infty$, $\FF(E)$ and $\FF(X)$ are necessarily nonisomorphic. We remark that if $X$ is nonseparable, then $\FF(E)$ is separable while $\FF(X)$ is nonseparable, and therefore the interesting case is when $X$ is separable. The following result provides a negative answer to the first question and a partial answer to the second in the case that $X$ fails the Dunford--Pettis property; in particular, it applies whenever $X$ is separable, infinite-dimensional and reflexive.

\begin{restatable}{corollary}{corthree}\label{inf}
    Let $X$ be a separable infinite-dimensional Banach space that fails the Dunford--Pettis property. Then for any $d\in \mathbb{N}$ and any nonempty $M\subset \rr^d$, $\FF(M)$ is not isomorphic to $\FF(X)$. In particular, this holds if $X=\ell_p$ for some $1<p<\infty$, or, more generally, if $X$ is a separable infinite-dimensional reflexive Banach space.
\end{restatable}

\noindent In \cite{isolipschitzspace}, Candido, C\'uth and Doucha ask whether $\FF(\ell_p)$ is isomorphic to $\left(\bigoplus\limits_{n=1}^\infty \FF(\ell_p^n) \right )_{\ell_1}$ for some (or even every) $1\leq p<\infty$ (\cite{isolipschitzspace}*{Question 8}). They also prove that their dual spaces are isomorphic (\cite{isolipschitzspace}*{Theorem 3.1}). As far as we know, this isomorphism problem for the Lipschitz-free spaces has remained open. The following result shows nonisomorphism for every $1<p<\infty$.

\begin{restatable}{corollary}{corfour}\label{l1sum}
    Let $1<p<\infty$. Then $\FF(\ell_p)$ and $\left(\bigoplus\limits_{n=1}^\infty \FF(\ell_p^n) \right)_{\ell_1}$ are not isomorphic.
\end{restatable}

\subsection{Background and motivation}

\noindent We now discuss some background and motivation. Recall that $\FF(M)$ separates the points of $\lip(M)$ and the closed unit ball $B_{\lip(M)}$ is compact in the initial topology $\sigma(\lip(M), \FF(M))$. Consequently, $\FF(M)^\star$ is isometrically isomorphic to $\lip(M)$. The map $\delta_M:M\to \FF(M)$ is an isometric embedding, and $\FF(M)$ is a Banach space that linearises the Lipschitz structure of $M$: given any metric space $N$ with basepoint $0_N$ and any Lipschitz map $f:M\to N$ such that $f(0_M)=0_N$, there is a unique bounded linear map $\hat{f}:\mathcal{F}(M)\to \mathcal{F}(N)$ which extends $f$ (in the sense that $\hat{f}\circ \delta_M=\delta_N\circ f$), and moreover $\lVert \hat{f} \rVert=\lVert f \rVert_{\Lip}$. \\

\noindent There has been much research on Lipschitz spaces and Lipschitz-free spaces over subsets of $\rr^n$, and some fundamental questions remain open. By Naor and Schechtman \cite{planarearthmover}, for $n\geq 2$, $\FF(\rr^n)$ is not isomorphic to a subspace of $\FF(\rr)$. Currently, it is unknown whether $\FF(\rr^2)$ and $\FF(\rr^3)$ are isomorphic. More generally, for distinct $m, n\geq 2$, it remains open whether $\FF(\rr^m)$ and $\FF(\rr^n)$ can be isomorphic; the analogous problem for $\lip(\rr^m)$ and $\lip(\rr^n)$ is also unresolved. In \cite{isolipschitzspace}, it was shown by Candido, C\'uth and Doucha that for every $d\in \mathbb{N}$, $\lip(\mathbb{Z}^d)$ and $\lip(\rr^d)$ are isomorphic. In \cite{Aliaga}, Aliaga generalised this result and showed that if $M\subset \mathbb{R}^d$ is non-porous (in particular, if $M$ is Lebesgue measurable with positive Lebesgue measure), then $\lip(M)$ is isomorphic to $\lip(\mathbb{R}^d)$. Aliaga asks whether the Lipschitz-free spaces $\FF(\rr\times \mathbb{Z})$ and $\FF(\rr^2)$ are isomorphic (see \cite{Aliaga}*{Question 5}). This question motivated us to consider the Dunford--Pettis property. It is known (see \cite{normedspaces}*{Proposition 3.14}) that $\FF(\rr\times\mathbb{Z})\sim \FF(\rr)\oplus_1 \FF(\mathbb{Z}^2)$, and that both $\FF(\rr)$ and $\FF(\mathbb{Z}^2)$ have the Dunford--Pettis property. For the former, $\FF(\rr)\cong L^1(\rr)$, while the latter in fact has the stronger Schur property; more generally, $\FF(M)$ has the Schur property for any complete, purely 1-unrectifiable metric space $M$ (see \cite{unrectifiable}). Therefore $\FF(\rr\times \mathbb{Z})$ has the Dunford--Pettis property. This suggested the possibility of distinguishing it from $\FF(\rr^2)$ by showing that $\FF(\rr^2)$ fails the Dunford--Pettis property. However, Theorem \ref{dunfordpettis} shows that $\FF(\rr^2)$ has the Dunford--Pettis property, so this property cannot be used to distinguish these spaces.

\subsection{Structure and overview}

\noindent The structure of this paper is as follows. In Section $2$, we collect results from the literature that will be used in the proofs of Theorem \ref{dunfordpettis} and Corollaries \ref{M}, \ref{inf} and \ref{l1sum}. In Section 3, we prove Theorem \ref{dunfordpettis} and then deduce Corollaries \ref{M}, \ref{inf} and \ref{l1sum} from it. We conclude in Section 4 with some open problems.

\section{Some results from the literature}

\noindent \textbf{Notation:} For Banach spaces $X$ and $Y$, we write $X \overset{c}{\hookrightarrow} Y$ to mean that $X$ is isomorphic to a complemented subspace of $Y$. We write $D_X=\{x\in X:\lVert x \rVert<1\}$ for the open unit ball of $X$, and $B_X=\{x\in X:\lVert x \rVert\leq 1\}$ for the closed unit ball of $X$. We write $X\cong Y$ to denote that $X$ and $Y$ are isometrically isomorphic, and $X\sim Y$ to denote that $X$ and $Y$ are isomorphic. For a metric space $(M, d)$, a point $x\in M$, and $r>0$, we write $\overline{B}(x, r)=\{y\in M:d(x, y)\leq r\}$ for the closed ball of radius $r$ in $M$ centred at $x$. 
If a metric space $M$ is contained in an ambient space with a distinguished point $0$ (for example, if $M$ is a subset of a Banach space), and we choose a different basepoint $q\in M$, we may sometimes write $\Lip_q(M)$ instead of $\lip(M)$ to avoid confusion. Unless otherwise stated, the metric on $\rr^n$ and all of its subsets is taken to be that derived from the Euclidean norm. \\

\noindent We now collect some results from the literature. The first is due to Kaufmann, and may be found as \cite{Kaufmann}*{Corollary 3.5}. 

\begin{proposition}\label{nonemptyinterior}
    Let $n\in \mathbb{N}$ and let $F\subset \rr^n$ have nonempty interior. Then $\FF(F)\sim \FF(\rr^n)$.
\end{proposition}

\noindent The next result gives, for a nonempty open convex subset $U$ of $\rr^n$, an explicit isometric representation of $\FF(U)$ as a quotient of $L^1(U;\rr^n)$. It may be found as the particular case $E=\rr^n$ of \cite{quotientrepresentation}*{Theorem 1.1}. In the following, $L^1(U;\rr^n)=\{(g_1, \cdots, g_n):g_i\in L^1(U)\}$ with the norm $\lVert (g_1, \cdots, g_n) \rVert=\displaystyle\int\limits_{U} \lVert (g_1(x), \cdots, g_n(x)) \rVert_2 \ dx$. Given $\textbf{g}\in L^1(U;\rr^n)$, we extend $\textbf{g}$ by $0$ outside $U$ to give an integrable field on $\rr^n$, and interpret $\mathrm{div}(\textbf{g})$ as the distributional divergence of this extension, an element of $\mathcal{D}'(\rr^n)$. The case $U=\rr^n$ was also obtained by Godefroy and Lerner in \cite{R^nquotient}*{Theorem 2.3}.

\begin{proposition}\label{representation}
    Let $n\in \mathbb{N}$, and let $U\subset \rr^n$ be nonempty, open and convex. Then $\FF(U)$ is isometrically isomorphic to the quotient space $L^1(U;\rr^n)/N$, where $N$ is the closed subspace of $L^1(U; \rr^n)$ consisting of all $\textbf{g}\in L^1(U;\rr^n)$ such that $\mathrm{div}(\textbf{g})=0$ in $\mathcal{D}'(\rr^n)$.
\end{proposition}

\noindent Before the next proposition, recall that for any open $U\subset \rr^n$, $L^1(U;\rr^n)^*$ is isometrically isomorphic to $L^\infty(U;\rr^n)$, with the duality pairing given by 
\[\langle \textbf{f}, \textbf{g}\rangle=\int\limits_U \langle \textbf{f}(x), \textbf{g}(x) \rangle_{\rr^n} \ dx, \ \qquad\forall \textbf{f}\in L^1(U;\rr^n), \ \forall\textbf{g}\in L^\infty(U;\rr^n).\] \noindent Under this identification, the next result identifies the annihilator $N^\perp$ with the subspace of gradients of Lipschitz functions. It can be obtained by combining Propositions 3.2 and 3.3 of \cite{quotientrepresentation}.

\begin{proposition}\label{nperp}
    Let $n\in \mathbb{N}$, $U\subset \rr^n$ be nonempty, open and convex, and $x\in U$ be a basepoint. Let $N$ be the closed subspace of $L^1(U; \rr^n)$ consisting of all $\textbf{g}\in L^1(U;\rr^n)$ such that $\mathrm{div}(\textbf{g})=0$ in $\mathcal{D}'(\rr^n)$. Then \[N^\perp=\{\nabla p:p\in \Lip_x(U)\}\subset L^\infty(U;\rr^n).\]
\end{proposition}

\noindent The following result characterises weakly null sequences in $L^\infty(\rr^n)$ and is a special case of \cite{Toland}*{Theorem 3.6}.

\begin{proposition}\label{weaklinfty}
    Let $(f_k)_{k=1}^\infty$ be a bounded sequence in $L^\infty(\rr^n)$. The following are equivalent.
    \begin{itemize}
        \item $f_k\cw 0$ in $L^\infty(\rr^n)$.
        \item For every strictly increasing sequence $j_1<j_2<j_3<\cdots$ of natural numbers, $\left \lVert \underset{1\leq k\leq m}{\min} \vert f_{j_k} \vert \right \rVert_\infty\to 0$ as $m\to \infty$.
    \end{itemize}
\end{proposition}

\noindent \textbf{Note:} This implies that if $f_k\cw 0$ in $L^\infty(\rr^n)$, then necessarily $\vert f_k \vert\cw 0$ in $L^\infty(\rr^n)$. \\

\noindent Before the next result, we briefly recall the definition of a doubling metric space.
\begin{definition}
     A metric space $M$ is \emph{doubling} if there exists $k\in \mathbb{N}$ such that for every $r>0$ and every $x\in M$, there exist $x_1, \cdots, x_k\in M$ with $\overline{B}(x, r)\subset \bigcup\limits_{i=1}^k \overline{B}(x_i, \frac{r}{2})$ (i.e. every closed ball can be covered by $k$ closed balls of half its radius).  
\end{definition}
\noindent The construction underlying the following result is the existence of gentle partitions of unity with respect to doubling subspaces, which was proved by Lee and Naor in \cite{LeeNaor}. It follows immediately from \cite{isolipschitzspace}*{Proposition 1.8}, using the fact that any subset of a doubling metric space is doubling. We note that the result holds more generally: \cite{isolipschitzspace} only assumes that $N$ is doubling.
\begin{proposition}[\cite{isolipschitzspace}*{Proposition 1.8}]\label{Complementation} Let $N$ and $M$ be nonempty subsets of a doubling metric space with $N\subset M$. Then $\mathcal{F}(N) \overset{c}{\hookrightarrow} \mathcal{F}(M)$ and $\lip(N) \overset{c}\hookrightarrow \lip(M)$. 
\end{proposition} 

\noindent There are various equivalent definitions of the Dunford--Pettis property. The one we shall use is as follows.

\begin{definition}\label{dunfordpettisdef}
        Let $X$ be a Banach space. $X$ has the \emph{Dunford--Pettis property} if for every weakly null sequence $(x_n)_{n=1}^\infty$ in $X$ and every weakly null sequence $(f_n)_{n=1}^\infty$ in $X^\star$, $f_n(x_n)\to 0$ as $n\to \infty$.
\end{definition}

\noindent \textbf{Note:} An equivalent formulation is that a Banach space $X$ has the Dunford--Pettis property if and only if for every Banach space $Y$ and every weakly compact $T:X\to Y$, $T$ maps weakly compact subsets of $X$ to norm compact subsets of $Y$. If $X$ is an infinite-dimensional reflexive space, then $B_X$ is weakly compact, so $\mathrm{id}_X:X\to X$ is weakly compact. However, since $B_X$ is not norm compact, we deduce that $X$ fails the Dunford--Pettis property. \\   

\noindent The next result collects two classical facts on the Dunford--Pettis property. After the statement, we briefly explain why both are true. 

\begin{lemma}\label{grothendieck}
    (i) Let $X$ be a Banach space with the Dunford--Pettis property, and $Y$ be another Banach space such that $Y\comp X$. Then $Y$ has the Dunford--Pettis property. \\
    (ii) Let $(X_n)_{n\geq 1}$ be a sequence of Banach spaces, each of which has the Dunford--Pettis property. Then $\left (\bigoplus\limits_{n=1}^\infty X_n \right )_{\ell_1}$ has the Dunford--Pettis property.
\end{lemma}

\noindent To see why $(i)$ holds, as the Dunford--Pettis property is an isomorphism invariant, we may assume that $X=Y\oplus_1 Z$ for some Banach space $Z$. Then if $\pi_Y:X\to Y$ denotes the projection onto $Y$, any weakly null sequence $(f_n)_{n=1}^\infty$ in $Y^\star$ may be extended to the weakly null sequence $(f_n\circ \pi_Y)_{n=1}^\infty$ in $X^\star$. So if $(y_n)_{n\geq 1}$ is weakly null in $Y$, it is also weakly null in $X$, and so $f_n(y_n)=(f_n\circ \pi_Y)(y_n)\to 0$ as $n\to \infty$. Thus the Dunford--Pettis property of $X$ implies the same for $Y$. A sketch of the proof of $(ii)$ is as follows. Suppose that $(x_k)_{k=1}^\infty=((x_{k, n})_{n=1}^\infty)_{k=1}^\infty$ is weakly null in $\left (\bigoplus\limits_{n=1}^\infty X_n \right )_{\ell_1}$ and $(f_k)_{k=1}^\infty=((f_{k, n})_{n=1}^\infty)_{k=1}^\infty$ is weakly null in $\left (\bigoplus\limits_{n=1}^\infty X_n^* \right )_{\ell_\infty}$. It follows that $\underset{k}{\sup}\sum\limits_{n>N}\lVert  x_{k, n} \rVert\to 0$ as $N\to \infty$. To see this, argue by contradiction and use a gliding hump argument together with suitable norming functionals on successive coordinate blocks to obtain a contradiction to weak nullity of $(x_k)_{k=1}^\infty$. For every fixed $n$, $(f_{k, n})_{k=1}^\infty$ is weakly null in $X_n^\star$ and $(x_{k, n})_{k=1}^\infty$ is weakly null in $X_n$, so $f_{k, n}(x_{k, n})\to 0$ as $k\to \infty$ by the Dunford--Pettis property of $X_n$. Since the tails $\sum\limits_{n>N} \lVert x_{k, n} \rVert$ can be taken as small as desired uniformly in $k$, and $\underset{k, n}{\sup}\lVert f_{k, n} \rVert=\underset{k}{\sup}\lVert f_k \rVert<\infty$, it follows that $f_k(x_k)=\sum\limits_{n=1}^\infty f_{k, n}(x_{k, n})\to 0$ as $k\to \infty$, as desired. \\

\noindent The next result is a consequence of the isometric lifting property for separable Banach spaces.

\begin{lemma}\label{lifting}
    Let $X$ be a separable Banach space failing the Dunford--Pettis property. Then the Lipschitz-free space $\FF(X)$ also fails the Dunford--Pettis property.
\end{lemma}

\begin{proof}
     Godefroy and Kalton proved that every separable Banach space $X$ has the isometric lifting property (see \cite{godefroykalton}*{Theorem 3.1}); in particular, $X\comp \FF(X)$. Therefore if $X$ is separable and fails the Dunford--Pettis property, then by Lemma \ref{grothendieck} (i), $\FF(X)$ must also fail the Dunford--Pettis property.
\end{proof}

\section{Proofs of the main results} 

\noindent We now prove Theorem \ref{dunfordpettis}, restated below for convenience. The core of the proof closely follows a beautiful construction of Bourgain, with some small modifications. Bourgain's construction can be found in the proof of \cite{bourgain}*{Proposition 2}. We encourage the reader to look at his argument; we use his notation at several points throughout to highlight the many similarities.

\thmone*

\begin{proof}
    We first make some simple reductions. Since the Dunford--Pettis property is an isomorphism invariant,  by Proposition \ref{nonemptyinterior}, it suffices to show that $\FF((-1, 1)^d)$ has the Dunford--Pettis property. Let $N$ be the closed subspace of $L^1((-1, 1)^d;\rr^d)$ consisting of all $\textbf{g}$ such that $\mathrm{div}(\textbf{g})=0$ in $\mathcal{D}'(\rr^d)$. By Proposition \ref{representation}, \[\FF((-1, 1)^d)\cong L^1((-1, 1)^d;\rr^d)/N,\] and by Proposition \ref{nperp}, it follows that the map \begin{equation}\label{iso} p\mapsto \nabla p:\lip((-1, 1)^d)\to N^\perp\end{equation} is an isometric isomorphism. Since the weak topology on $N^\perp$ is precisely the restriction to $N^\perp$ of the weak topology on $L^\infty((-1, 1)^d;\rr^d)$, it suffices by Definition \ref{dunfordpettisdef} to show that the following holds. \\
    
    \noindent \underline{\textbf{(A):}} Let $(\textbf{f}_k)_{k=1}^\infty$ be a sequence in $L^1((-1, 1)^d;\rr^d)$, $(p_k)_{k=1}^\infty$ be a sequence in $\lip((-1, 1)^d)$ and assume that both of the following hold.
    \begin{itemize}
        \item\begin{equation}\label{cond1} \forall p\in \lip((-1, 1)^d), \ \displaystyle\int\limits_{(-1, 1)^d} \textbf{f}_k \cdot \nabla p \ dx\to 0 \text{ as } k\to \infty. \end{equation}
        \item \begin{equation}\label{cond2} \nabla p_k \cw 0 \text{ in } L^\infty((-1, 1)^d ;\rr^d) \text{ as } k\to \infty. \end{equation}
    \end{itemize}
    \noindent Then \[\displaystyle\int\limits_{(-1, 1)^d} \textbf{f}_k \cdot \nabla p_k \ dx\to 0\] as $k\to \infty$. \\
    
    \noindent We argue by contradiction and assume that \textbf{(A)} is false. Then, by passing to a subsequence and replacing each $p_k$ by $-p_k$ when necessary, we can find $\delta>0$ and sequences $(\textbf{f}_k)_{k=1}^\infty\subset L^1((-1, 1)^d;\rr^d)$ and $(p_k)_{k=1}^\infty \subset \lip((-1, 1)^d)$ satisfying \eqcref{cond1} and \eqcref{cond2}, such that \[\forall k, \ \int\limits_{(-1, 1)^d} \textbf{f}_k \cdot \nabla p_k \ dx>\delta.\] Since the weakly null sequence $(\textbf{f}_k+N)_{k=1}^\infty$ is bounded in $L^1((-1, 1)^d;\rr^d)/N$, after replacing each $\textbf{f}_k$ by $c\textbf{f}_k$ for some suitable fixed $c>0$ and adjusting $\delta$, we may assume that $\lVert \textbf{f}_k+N \rVert<1$ for every $k$. Recall that for any Banach space $X$ and closed subspace $Y$ of $X$, $q(D_X)=D_{X/Y}$ where $q:X\to X/Y$ is the quotient map. Thus, by replacing $\textbf{f}_k$ with a suitable representative of the coset $\textbf{f}_k+N$ (which does not change any of the pairings with $\nabla p$ for $p\in \lip((-1, 1)^d)$), we may without loss of generality assume that for all $k$, $\lVert \textbf{f}_k \rVert_{L^1((-1, 1)^d;\rr^d)}\leq 1$. Let $E: \lip((-1, 1)^d)\to \{g\in\lip(\rr^d):\mathrm{supp}(g)\subset [-2, 2]^d\}$ be a bounded linear extension operator. To justify the existence of $E$, let $r:\rr^d\to [-1, 1]^d$ be a Lipschitz retraction, and $\chi:\rr^d\to \rr$ be Lipschitz such that $\chi=1$ on $[-1, 1]^d$ and $\chi=0$ outside $[-2, 2]^d$. Then for $p\in \lip((-1, 1)^d)$, set $(Ep)(x)=\chi(x) \overline{p}(r(x))$, where $\overline{p}$ denotes the unique Lipschitz extension of $p$ to $[-1, 1]^d$. \\
    
    \noindent We have that $p_k\cw 0$ in $\lip((-1, 1)^d)$ by the isometric identification \eqcref{iso}. This implies that $Ep_k\cw 0$ in $\lip(\rr^d)$, and therefore $\nabla (Ep_k)\cw 0$ in $L^\infty(\rr^d;\rr^d)$ (since both $E$ and $h\mapsto \nabla h:\lip(\rr^d)\to L^\infty(\rr^d;\rr^d)$ are bounded and hence $w$-to-$w$ continuous). The weakly null sequence $(Ep_k)_{k=1}^\infty$ is bounded in $\lip(\rr^d)$, so after multiplying the sequence $(p_k)_{k=1}^\infty$ by a fixed positive constant and adjusting $\delta$, we may assume that $\lVert Ep_k \rVert_{\Lip}\leq 1$ for every $k$. Then, by extending $\textbf{f}_k$ by $0$ outside $(-1, 1)^d$ and replacing $p_k$ with $Ep_k$, we can find sequences $(\textbf{f}_k)_{k=1}^\infty\subset B_{L^1(\rr^d;\rr^d)}$ and $(p_k)_{k=1}^\infty\subset B_{\lip(\rr^d)}$ such that the following hold. The first bullet point follows from \eqcref{cond1}, since $\textbf{f}_k$ is extended by $0$ and each $p\in \lip(\rr^d)$ restricts to an element of $\lip((-1, 1)^d)$. The fourth bullet point follows from $Ep=p$ on $(-1, 1)^d$ and the adjustment of $\delta$ after scaling.
    \begin{itemize}
        \item \begin{equation}\label{2.3}
            \forall p\in \lip(\rr^d), \dd \textbf{f}_k\cdot \nabla p \ dx \to 0 \text{ as }k\to \infty.
        \end{equation}
        \item \begin{equation*}\label{2.4}
            \nabla p_k \cw 0 \text{ in }L^\infty(\rr^d;\rr^d) \text{ as }k\to \infty.
        \end{equation*}
        \item \begin{equation*} \label{2.5}
            \forall k, \ \mathrm{supp}(p_k)\subset [-2, 2]^d.
        \end{equation*}

        \item \begin{equation*} \label{2.6}
            \forall k, \ \dd \textbf{f}_k \cdot \nabla p_k \ dx>\delta.
        \end{equation*}
    \end{itemize}

\noindent From now on, $N$ will denote the closed subspace of $L^1(\rr^d;\rr^d)$ consisting of the divergence-free integrable vector fields. We have that $p\mapsto \nabla p:\lip(\rr^d)\to L^\infty(\rr^d;\rr^d)$ is an isometric isomorphism onto $N^\perp$. For a Lipschitz function $p$ on $\rr^d$, $\left \lVert \nabla p \right \rVert$ will denote the Euclidean norm of the gradient of $p$, which exists almost everywhere on $\rr^d$. Since the weak topology on $N^\perp$ agrees with the restriction to $N^\perp$ of the weak topology on $L^\infty(\rr^d;\rr^d)$, it follows that $p_k\cw 0$ in $\lip(\rr^d)$. For any $x\in \rr^d$, $p\mapsto p(x)$ defines an element of $\lip(\rr^d)^*$, so we have that $p_k\to 0$ pointwise on $\rr^d$. Since each $p_k$ is $1$-Lipschitz and is $0$ outside of $[-2, 2]^d$, pointwise convergence to $0$ on arbitrarily fine finite nets in $[-2, 2]^d$ implies that $p_k\to 0$ uniformly on $\rr^d$. For each $1\leq i\leq d$, $D_i p_k\cw 0$ in $L^\infty(\rr^d)$, and so by the note following Proposition \ref{weaklinfty}, we have that $\vert D_i p_k|\cw 0$ in $L^\infty(\rr^d)$. Therefore $\sum\limits_{i=1}^d \vert D_i p_k \vert\cw 0$ in $L^\infty(\rr^d)$. Since $0\leq\lVert \nabla p_k \rVert\leq \sum\limits_{i=1}^d \vert D_i p_k \vert$ almost everywhere, the criterion in Proposition \ref{weaklinfty} implies that \[\label{2.7}\refstepcounter{equation}\lVert \nabla p_k \rVert\cw 0 \text{ in }L^\infty(\rr^d). \tag{\theequation}\] 

\noindent Let $\phi\in C_c^\infty(\rr^d)$ be such that $\phi\geq 0$, $\mathrm{supp}(\phi)\subset \overline{B}(0, 1)$ and $\lVert \phi \rVert_1=1$. For $\varepsilon>0$, $x\in \rr^d$, let $\phi_{\varepsilon}(x)=\frac{1}{\varepsilon^d}\phi\left(\frac{x}{\varepsilon}\right)$. If $\textbf{f}\in L^1(\rr^d;\rr^d)$ and $p\in \lip(\rr^d)$, then $p\star \phi_\varepsilon$ is Lipschitz with $\nabla (p\star \phi_\varepsilon)=(\nabla p)\star \phi_\varepsilon$. By the Lebesgue differentiation theorem, it follows that $\nabla(p\star \phi_\varepsilon)\to \nabla p$ almost everywhere on $\rr^d$ as $\varepsilon\to 0$. Since $\nabla (p\star \phi_\varepsilon)$ is bounded in norm pointwise by $\lVert \nabla p \rVert_\infty$, dominated convergence gives that \[\label{2.8}\refstepcounter{equation}\dd \textbf{f} \cdot \nabla(p\star \phi_\varepsilon) \ dx\to \dd \textbf{f} \cdot \nabla p \ dx \text{ as }\varepsilon\to 0. \tag{\theequation}\]

\noindent\underline{\textbf{An inductive construction}} \\

\noindent From here, we follow Bourgain's argument closely, with some small changes. Let $(\varepsilon_k)_{k=1}^\infty \subset (0, \infty)$ be such that $\sum\limits_{k=1}^\infty \varepsilon_k<\delta/2$. We will inductively find $n_1<n_2<\cdots$ in $\mathbb{N}$, mollification parameters $(\eta_k)_{k=1}^\infty$, smooth Lipschitz functions $y_k=p_{n_k}\star \phi_{\eta_k}$, smooth Lipschitz cut-off functions $\psi_k:\mathbb{R}^d\to [0, 1]$ and a decreasing sequence $(N_k)_{k=1}^\infty$ of infinite subsets of $\mathbb{N}$. These will have the following properties. First, \[n_k\in N_{k-1} \text{ for every } k\geq 2.\] Second, for every $k$, \[\psi_k=1 \text{ on } \{x:\lVert \nabla y_k(x) \rVert\geq 2\varepsilon_k\}, \qquad \psi_k=0 \text{ on } \{x:\lVert \nabla y_k(x) \rVert\leq \varepsilon_k\}.\] Third, writing $P_0=1$ and  \[P_k=(1-\psi_1)\cdots(1-\psi_k) \qquad (k\in \mathbb{N}),\] we will have
\[\forall k\in \mathbb{N}, \forall n\in N_{k}, \forall \eta>0, \left \vert\int_{\mathbb R^d}\textbf{f}_n\cdot\nabla\left(\psi_kP_{k-1}(p_n\star\phi_\eta)\right)\ dx \right\vert<\varepsilon_{k}.
\] Fourth, \[\lVert \nabla P_{k-1} \rVert_\infty \lVert y_k \rVert_\infty<\varepsilon_k \text{ for every }k\in \mathbb{N}.\] Finally, for every $k\in \mathbb{N}$, \[\dd \textbf{f}_{n_k}\cdot \nabla y_k \ dx>\delta.\]  Call this collection of properties (H). \\

\noindent The construction will ensure that, upon setting \[x_k=(1-\psi_1)\cdots(1-\psi_{k-1})y_k \qquad(k\in \mathbb{N}),\] the sequence $(\nabla x_k)_{k=1}^\infty$ is dominated by the unit vector basis of $c_0$, while the diagonal pairings \[\dd \textbf{f}_{n_k}\cdot \nabla x_k \ dx\]remain bounded away from $0$.  \\ 

\noindent\emph{Base case:} By Mazur's theorem, \eqcref{2.7} implies that there exist a finite subset $D_1$ of $\mathbb{N}$ and nonnegative coefficients $(\lambda_m)_{m\in D_1}$ such that $\sum\limits_{m\in D_1} \lambda_m=1$ and $\left \lVert \sum\limits_{m\in D_1}\lambda_m \lVert \nabla p_m \rVert \right \rVert_\infty<\varepsilon_1^2$. We may without loss of generality assume that $\lambda_m>0$ for all $m\in D_1$. By \eqcref{2.8}, we may choose $\eta_1>0$ so that, setting $q_m=p_m\star \phi_{\eta_1}$ for each $m\in D_1$, we have \[\label{2.9}\refstepcounter{equation}\dd \textbf{f}_m \cdot \nabla q_m \ dx>\delta \qquad \text{for every } m\in D_1. \tag{\theequation}\] Since $\nabla q_m=(\nabla p_m)\star \phi_{\eta_1}$ for each $m\in D_1$, we have, for every $x\in \rr^d$, \[\lVert \nabla q_m(x)\rVert\leq \left( \lVert\nabla p_m\rVert\star \phi_{\eta_1} \right )(x).\] Thus \[\sum\limits_{m\in D_1}\lambda_m \lVert \nabla q_m \rVert\leq \left( \sum\limits_{m\in D_1} \lambda_m \lVert \nabla p_m \rVert\right )\star \phi_{\eta_1}\] pointwise, and therefore \[\left \lVert \sum\limits_{m\in D_1} \lambda_m \lVert \nabla q_m \rVert \right \rVert_\infty\leq \left \lVert \sum\limits_{m\in D_1} \lambda_m \lVert \nabla p_m \rVert \right \rVert_\infty<\varepsilon_1^2.\] For each $m\in D_1$, $q_m\in C_c^\infty(\rr^d)$, and therefore $C_{1, m}=\{x:\lVert \nabla q_m(x) \rVert\geq 2\varepsilon_1\}$ and $C_{2, m}=\{x: \lVert\nabla q_m(x) \rVert\leq \varepsilon_1\}$ are disjoint closed sets, with $C_{1, m}$ compact. Thus $d(C_{1, m}, C_{2, m})>0$, so we may choose a smooth Lipschitz function $\tau_m:\rr^d\to [0, 1]$ such that $\tau_m=1$ on $C_{1, m}$ and $\tau_m=0$ on $C_{2, m}$. For any $x\in \rr^d$, 
\[\sum\limits_{m\in D_1} \lambda_m \tau_m(x)\leq \frac{1}{\varepsilon_1}\sum\limits_{m\in D_1} \lambda_m \lVert \nabla q_m(x) \rVert,\] and therefore $\lVert \sum\limits_{m\in D_1} \lambda _m \tau_m\rVert_\infty<\varepsilon_1$. \\

\noindent For $n\in\mathbb N$, $\eta>0$ and $m\in D_1$, set
\[I_{n,m}(\eta)=\int_{\mathbb R^d}
    \mathbf f_n\cdot
    \nabla\left(\tau_m(p_n\star\phi_\eta)\right)\,dx.\]
Put \[a=\left\lVert\sum\limits_{m\in D_1}\lambda_m\tau_m\right\rVert_\infty
<\varepsilon_1, \qquad C=\sum_{m\in D_1}\lambda_m\lVert \nabla \tau_m \rVert_\infty.\]
By the product rule and the inequalities
\[\lVert p_n\star\phi_\eta\rVert_\infty\leq\lVert p_n\rVert_\infty, \qquad \lVert (\nabla p_n)\star\phi_\eta\rVert _\infty \leq\|\nabla p_n\|_\infty,
\]
we have
\[
\begin{aligned}
\sum_{m\in D_1}\lambda_m \sup_{\eta>0}|I_{n,m}(\eta)|
&\leq 
\lVert \nabla p_n\rVert_\infty \int_{\mathbb R^d}  \lVert \textbf{f}_n\rVert \sum_{m\in D_1}\lambda_m\tau_m\ dx +\lVert p_n\rVert_\infty \int_{\mathbb R^d}
    \lVert \textbf{f}_n\rVert
    \sum_{m\in D_1}\lambda_m\lVert \nabla\tau_m\rVert\ dx \\
&\leq \lVert \textbf{f}_n\rVert_{L^1(\rr^d;\rr^d)}\left(a\lVert\nabla p_n\rVert_\infty+C\lVert p_n\rVert_\infty\right) \\
&\leq a+C\lVert p_n\rVert_\infty.
\end{aligned}
\] \noindent In the last inequality we used that $\lVert \textbf{f}_n \rVert_{L^1(\rr^d;\rr^d)}\leq 1$ and $\lVert \nabla p_n \rVert_\infty\leq 1$.
Since $\lVert p_n\rVert_\infty\to 0$ and $a<\varepsilon_1$, it follows that for all sufficiently large $n$,
\[\sum_{m\in D_1}\lambda_m \sup_{\eta>0} \left \vert\int_{\mathbb R^d}    \textbf{f}_n\cdot \nabla\left(\tau_m(p_n\star\phi_\eta)\right )\ dx \right \vert<\varepsilon_1.\]
\noindent For each such $n$, there exists $m\in D_1$ such that \[\sup_{\eta>0} \left \vert\int_{\mathbb R^d}    \textbf{f}_n\cdot \nabla\left(\tau_m(p_n\star\phi_\eta)\right )\ dx \right \vert<\varepsilon_1.\] Since $D_1$ is finite, by the pigeonhole principle we can choose $n_1\in D_1$ and an infinite set $N_1\subset \mathbb{N}$ such that 
\[\forall n\in N_1, \forall \eta>0, \left \vert  \dd \textbf{f}_n \cdot \nabla \left( \tau_{n_1}(p_n\star \phi_\eta) \right ) \ dx\right \vert<\varepsilon_1.\]
\noindent Set $\psi_1=\tau_{n_1}$ and $y_1=q_{n_1}$. This constructs $n_1$, $\eta_1$, $y_1$, $\psi_1$ and $N_1$, and we note that the properties (H) hold. The first property is vacuous. For the second, note that \[\psi_1=\tau_{n_1}=1 \text{ on }\{x:\lVert \nabla q_{n_1}(x) \rVert\geq 2\varepsilon_1\}=\{x:\lVert \nabla y_1(x) \rVert\geq 2\varepsilon_1\},\] and similarly $\psi_1=0$ on $\{x:\lVert \nabla y_1(x) \rVert\leq \varepsilon_1\}$. The third property was justified above, since $P_0=1$ and $\psi_1=\tau_{n_1}$. The fourth property holds since $\nabla P_0=0$. The fifth property holds by \eqcref{2.9} since $y_1=q_{n_1}$ and $n_1\in D_1$. \\

\noindent\emph{Inductive step:} Assume that the construction has been carried out up to step $k$ and properties (H) hold up to step $k$. Since $1-\psi_j$ is bounded and Lipschitz for each $1\leq j\leq k$, $P_k=(1-\psi_1)\cdots(1-\psi_k)$ is Lipschitz. Since $\lVert p_n \rVert_\infty\to 0$, there exists $J_{k+1}\in \mathbb{N}$ with $J_{k+1}>n_k$ such that \[\left \lVert \nabla P_k\right \rVert_\infty \lVert p_n \rVert_\infty<\varepsilon_{k+1}\text{ for all }n>J_{k+1}.\] As before, by applying Mazur's theorem to the weakly null subsequence $(\lVert \nabla p_n \rVert)_{\substack{n\in N_k \\ n>J_{k+1}}}$, we can find a finite set $D_{k+1}\subset N_k\setminus\{1, 2, \cdots, J_{k+1}\}$ and strictly positive coefficients $(\lambda_m)_{m\in D_{k+1}}$ with $\sum\limits_{m\in D_{k+1}} \lambda_m=1$, such that $\left \lVert \sum\limits_{m\in D_{k+1}} \lambda_m \lVert \nabla p_m \rVert\right \rVert_\infty<\varepsilon_{k+1}^2$. By \eqcref{2.8}, we may choose $\eta_{k+1}>0$ so that, setting $q_m=p_m \star \phi_{\eta_{k+1}}$ for each $m\in D_{k+1}$, we have \[\label{2.10}\refstepcounter{equation}\dd \textbf{f}_m \cdot \nabla q_m \ dx>\delta \qquad \text{ for every } m\in D_{k+1}. \tag{\theequation}\] As before, for each $m\in D_{k+1}$, let $\tau_m:\rr^d\to [0, 1]$ be smooth and Lipschitz such that \[\tau_m=1 \text{ on } \{x:\lVert \nabla q_m(x) \rVert\geq 2\varepsilon_{k+1}\}, \qquad \tau_m=0 \text{ on } \{x:\lVert \nabla q_m(x) \rVert\leq \varepsilon_{k+1}\}.\] Then, exactly as in the base case, $\left \lVert \sum\limits_{m\in D_{k+1}}\lambda_m \tau_m \right\rVert_{\infty}<\varepsilon_{k+1}$. For $n\in N_k\setminus\{1,\cdots,J_{k+1}\}$, $\eta>0$, and $m\in D_{k+1}$, set \[I_{n,m}(\eta)=\int_{\mathbb R^d}
\mathbf f_n\cdot\nabla\left(\tau_m P_k(p_n\star\phi_\eta)\right)\ dx.\]
Put\[C=\sum_{m\in D_{k+1}}\lambda_m\left \lVert \nabla \tau_m\right \rVert_\infty.\]
By the product rule,
\[\begin{aligned}
\vert I_{n,m}(\eta)\vert
&\leq \int_{\mathbb R^d}\lVert \textbf{f}_n \rVert\ \tau_m P_k
\lVert(\nabla p_n)\star\phi_\eta\rVert \ dx \\
&\quad +\int_{\mathbb R^d} \lVert \textbf{f}_n \rVert\,\tau_m \lVert\nabla P_k\rVert\, |p_n\star\phi_\eta|\ dx \\
&\quad +\int_{\mathbb R^d} \lVert \textbf{f}_n \rVert\,\lVert \nabla\tau_m\rVert P_k\, |p_n\star\phi_\eta|\,dx.
\end{aligned}
\]
We have 
\[\lVert(\nabla p_n)\star\phi_\eta\rVert_\infty \leq \lVert \nabla p_n\rVert_\infty, \qquad \lVert p_n\star\phi_\eta\rVert _\infty \leq
\lVert p_n\rVert_\infty.
\]
Hence, using that $0\leq P_k\leq 1$, we obtain
\[
\begin{aligned}
\sum_{m\in D_{k+1}}\lambda_m
\sup_{\eta>0}\vert I_{n,m}(\eta)\vert
&\leq \lVert \nabla p_n\rVert_\infty \int_{\mathbb R^d} \lVert \textbf{f}_n \rVert \sum_{m\in D_{k+1}}\lambda_m\tau_m\ dx \\
&\quad +\lVert \nabla P_k\rVert_\infty\lVert p_n\rVert_\infty \int_{\mathbb R^d} \lVert \textbf{f}_n \rVert \sum_{m\in D_{k+1}}\lambda_m\tau_m\ dx \\
&\quad +\lVert p_n\rVert_\infty \int_{\mathbb R^d} \lVert \textbf{f}_n \rVert \sum_{m\in D_{k+1}}\lambda_m \lVert \nabla\tau_m\rVert\ dx \\
&\leq \lVert \textbf{f}_n\rVert_{L^1(\rr^d;\rr^d)} \left[\left(\lVert \nabla p_n\rVert_\infty+\lVert \nabla P_k\rVert_\infty\lVert p_n\rVert_\infty\right)\left\lVert\sum_{m\in D_{k+1}}\lambda_m\tau_m\right\rVert_\infty+C\lVert p_n\rVert_\infty\right].
\end{aligned}
\]
Using
\[\lVert \textbf{f}_n\rVert_{L^1(\rr^d;\rr^d)}\leq 1, \qquad \lVert \nabla p_n\rVert_\infty\leq 1,\] and writing \[a:=\left\lVert\sum_{m\in D_{k+1}}\lambda_m\tau_m\right\rVert_\infty<\varepsilon_{k+1},
\]
we obtain
\[\sum_{m\in D_{k+1}}\lambda_m \sup_{\eta>0}\vert I_{n,m}(\eta)\vert< a+\left(a\lVert\nabla P_k\rVert_\infty+C
\right)\lVert p_n\rVert_\infty.
\]
Since $a<\varepsilon_{k+1}$ and $\|p_n\|_\infty\to 0$, it follows that for all sufficiently large $n\in N_k$,
\[\sum_{m\in D_{k+1}}\lambda_m\sup_{\eta>0}|I_{n,m}(\eta)|<\varepsilon_{k+1}.\]
Hence, by the pigeonhole principle, there exist $n_{k+1}\in D_{k+1}$ and an infinite set $N_{k+1}\subset N_k\setminus \{1, 2, \cdots, J_{k+1}\}$ such that
\[\forall n\in N_{k+1}, \forall \eta>0, \left \vert\int_{\mathbb R^d}\textbf{f}_n\cdot\nabla\left(\tau_{n_{k+1}}P_k(p_n\star\phi_\eta)\right)\ dx \right\vert<\varepsilon_{k+1}.
\]

\noindent Set $y_{k+1}=q_{n_{k+1}}$ and $\psi_{k+1}=\tau_{n_{k+1}}$. \\

\noindent This constructs $n_{k+1}$, $\eta_{k+1}$, $y_{k+1}$, $\psi_{k+1}$ and $N_{k+1}$, and we again note that the properties (H) hold up to step $k+1$. The first property holds since \[n_{k+1}\in D_{k+1}\subset N_k.\] For the second property, $\psi_{k+1}=\tau_{n_{k+1}}=1$ on \[\{x: \lVert \nabla q_{n_{k+1}}(x) \rVert \geq 2\varepsilon_{k+1}\}=\{x: \lVert \nabla y_{k+1}(x) \rVert\geq 2\varepsilon_{k+1}\}\] and similarly $\psi_{k+1}=0$ on \[\{x:\lVert \nabla y_{k+1}(x) \rVert\leq \varepsilon_{k+1}\}.\] The third property was justified above, since $\psi_{k+1}=\tau_{n_{k+1}}$. For the fourth property, \[\lVert \nabla P_k \rVert_\infty \lVert y_{k+1}\rVert_\infty\leq \lVert \nabla P_k \rVert_\infty \lVert p_{n_{k+1}} \rVert_\infty<\varepsilon_{k+1},\] since $n_{k+1}>J_{k+1}$. The fifth property holds by \eqcref{2.10} since $y_{k+1}=q_{n_{k+1}}$ and $n_{k+1}\in D_{k+1}$. \\

\noindent\underline{\textbf{Domination by the unit vector basis of $c_0$}} \\

\noindent Set $x_1=y_1$, and for $k\geq 1$, set $x_{k+1}= P_k y_{k+1}$. Each $x_k$ is a smooth, bounded Lipschitz function on $\rr^d$. Let $t\in \rr^d$. We bound $\sum\limits_{k=1}^\infty \lVert \nabla x_k(t) \rVert$. For $k\geq 1$, we have
\[\nabla x_{k+1}(t)=\nabla P_k(t)y_{k+1}(t)+P_k(t)\nabla y_{k+1}(t).\]
\noindent For every $k\geq 1$, the fourth property of (H) gives \[\lVert \nabla P_k\rVert_\infty \lVert y_{k+1} \rVert_\infty<\varepsilon_{k+1}.\] Suppose that $\lVert \nabla y_j(t)\rVert\geq 2\varepsilon_j$ for some $j\in \mathbb{N}$. Let $k_t$ be the least such $j$. Then for every $k<k_t$, $\lVert \nabla y_k(t) \rVert<2\varepsilon_k$, and $\psi_{k_t}(t)=1$ by the second property of (H). Therefore $P_{k}(t)=0$ for every $k\geq k_t$, and hence, using that $\lVert \nabla y_{k_t} \rVert_\infty\leq \lVert \nabla p_{n_{k_t}} \rVert_{\infty}\leq 1$,
\[\sum\limits_{k=1}^\infty \lVert P_{k-1}(t)  \nabla y_k(t) \rVert\leq \sum\limits_{k=1}^{k_t-1} 2\varepsilon_k+\lVert \nabla y_{k_t} \rVert_\infty\leq 1+2\sum_{k=1}^\infty \varepsilon_k.\] \noindent Otherwise, $\lVert \nabla y_k(t) \rVert<2\varepsilon_k$ for every $k$, and the same bound holds. Therefore
\[\forall t\in \rr^d, \qquad\sum_{k=1}^\infty \lVert\nabla x_k(t)\rVert\leq 1+3\sum_{k=1}^\infty \varepsilon_k:=K<\infty.\]

\noindent This implies that for every $a=(a_k)_{k=1}^\infty\in c_{00}$,
\[\left\lVert \sum\limits_{k=1}^\infty a_k \nabla x_k \right \rVert\leq K \lVert a \rVert_\infty.\] \noindent For each $k$, let $\tilde{x}_k:=x_k-x_k(0)$. Then $\tilde{x}_k\in \lip(\rr^d)$ and $\nabla x_k=\nabla \tilde{x}_k\in N^\perp$. The preceding estimate therefore implies the existence of a bounded linear operator $\tilde{T}:(c_{00}, \lVert \cdot \rVert_\infty)\to N^\perp$ with $\tilde{T}e_n=\nabla x_n$ for all $n$. By density we may extend $\tilde{T}$ to a bounded linear operator $T:c_0\to N^\perp$ with $Te_n=\nabla x_n$ for all $n$. \\

\noindent\underline{\textbf{Lower bound on the diagonal pairings}} \\

\noindent To simplify notation, for $\textbf{f}\in L^1(\rr^d;\rr^d)$ and a Lipschitz function $p:\rr^d\to \rr$, we write $\langle \textbf{f}, p\rangle$ to denote $\dd \textbf{f}\cdot \nabla p \ dx$ (this pairing is unchanged by adding a constant to $p$). For $k\in \mathbb{N}$, observe that
\[1-(1-\psi_1)\cdots(1-\psi_{k-1})=\sum\limits_{1\leq j<k}\psi_j \prod_{1\leq i<j}(1-\psi_i).\] \noindent By the first property of (H) and the nestedness of $(N_l)_{l=1}^\infty$, $n_k\in N_j$ for every $j<k$, since $n_k\in N_{k-1}\subset \cdots \subset N_j$. Thus, for every $k\in \mathbb{N}$, using the third property of (H),
\[\begin{aligned}
    \left \vert \langle \textbf{f}_{n_k}, y_k-x_k \rangle\right \vert&\leq \sum_{1\leq j<k}\left\vert \left\langle \textbf{f}_{n_k}, \psi_j \left (\prod_{1\leq i<j}(1-\psi_i)\right )y_k \right \rangle \right \vert\\
    &=\sum_{1\leq j<k}\left\vert\left \langle\textbf{f}_{n_k}, \psi_j  P_{j-1} \left(p_{n_k}\star \phi_{\eta_k} \right ) \right \rangle \right \vert \\
    &<\sum_{1\leq j<k}\varepsilon_j \\
    &<\delta/2.
\end{aligned}\] \noindent By the fifth property of (H), $\langle \textbf{f}_{n_k}, y_k \rangle=\langle\textbf{f}_{n_k}, q_{n_k} \rangle>\delta$. We deduce that for every $k$, \[ \langle\textbf{f}_{n_k}, x_k \rangle  \geq  \langle \textbf{f}_{n_k}, y_k \rangle - \left \vert \langle \textbf{f}_{n_k}, y_k-x_k \rangle\right \vert> \delta-\delta/2= \delta/2.\] \\

\noindent\underline{\textbf{Conclusion of the proof}} \\

\noindent For $k, l\in \mathbb{N}$, let $a_{k, l}=\langle \textbf{f}_{n_k}, x_l \rangle$. Then $ a_{k, k} >\delta/2$ for every $k$. Let $\sigma=(\sigma_l)_{l=1}^\infty\in \{-1, 1\}^{\mathbb{N}}$ be a sequence of signs. Since $T$ is bounded, $\left (\sum\limits_{l=1}^m \sigma_l \nabla x_l\right )_{m=1}^\infty$ is a bounded sequence in $N^{\perp}$, uniformly in $\sigma$. Therefore, since $(\textbf{f}_n)_{n\geq 1}$ is bounded, $\underset{k\in \mathbb{N}}{\sup}\ \underset{\sigma\in \{-1, 1\}^\mathbb{N}}{\sup} \ \underset{m\in \mathbb{N}}{\sup} \left \vert \sum\limits_{l=1}^m \sigma_l a_{k, l} \right \vert<\infty$, which implies that $\underset{k\in \mathbb{N}}{\sup}\sum\limits_{l=1}^\infty \vert a_{k, l} \vert<\infty$. Fix $\sigma\in \{-1, 1\}^\mathbb{N}$. Since $N^{\perp}$ is the dual of the separable Banach space $L^1(\rr^d;\rr^d)/N$, by Banach-Alaoglu and $w^\star$-metrisability of norm-bounded sets of $N^\perp$, there exist $m_1<m_2<\cdots$ in $\mathbb{N}$ and $p_\sigma\in \lip(\rr^d)$ such that $\sum\limits_{l=1}^{m_j} \sigma_l \nabla x_l \overset{w^*}{\to} \nabla p_{\sigma}$ in $N^\perp$ as $j\to \infty$. Pairing this $w^\star$-convergence against $\textbf{f}_{n_k}+N$ and using absolute summability of $(a_{k, l})_{l=1}^\infty$, we deduce for every $k$ that $\langle \textbf{f}_{n_k}, p_\sigma \rangle=\sum\limits_{l=1}^\infty \sigma_l a_{k, l}$. By \eqcref{2.3}, we deduce that for every $\sigma\in \{-1, 1\}^{\mathbb{N}}$, $\sum\limits_{l=1}^\infty \sigma_l a_{k, l}\to 0$ as $k\to \infty$. By a gliding hump argument identical to the one used in the proof of the Schur property of $\ell_1$, this implies that $\sum\limits_{l=1}^\infty \vert a_{k, l} \vert\to 0$ as $k\to \infty$. This contradicts that $ a_{k, k} >\delta/2$ for every $k$, and completes the proof. 

\end{proof}

\noindent We obtain three corollaries of this result. The first states that the Lipschitz-free space over any nonempty subset of $\rr^d$ has the Dunford--Pettis property.

\restatecor{2}{\cortwo}

\begin{proof}
    By Theorem \ref{dunfordpettis}, $\FF(\rr^d)$ has the Dunford--Pettis property, and by Proposition \ref{Complementation}, $\FF(M)\comp \FF(\rr^d)$. Therefore, by Lemma \ref{grothendieck} (i), $\FF(M)$ has the Dunford--Pettis property. 
\end{proof}

\noindent The following result implies as a special case that $\FF(\rr^2)$ is not isomorphic to $\FF(\ell_2)$. 

\restatecor{3}{\corthree}

\begin{proof}
   \noindent By Lemma \ref{lifting}, $\FF(X)$ fails the Dunford--Pettis property. By Corollary \ref{M}, $\FF(M)$ has the Dunford--Pettis property, and hence $\FF(M)$ is not isomorphic to $\FF(X)$.
\end{proof}

\noindent In \cite{isolipschitzspace}, Candido, C\'uth and Doucha ask whether for some $1\leq p<\infty$, $\FF(\ell_p)$ is isomorphic to $\left(\bigoplus\limits_{n=1}^\infty \FF(\ell_p^n) \right )_{\ell_1}$ (\cite{isolipschitzspace}*{Question 8}).  The following result shows that  for every $1<p<\infty$, these spaces are nonisomorphic.

\restatecor{4}{\corfour}

\begin{proof}
    By Lemma \ref{lifting}, $\FF(\ell_p)$ fails the Dunford--Pettis property. For every $n$, $\ell_p^n$ is bilipschitz equivalent to $\ell_2^n$, so $\FF(\ell_p^n)\sim \FF(\rr^n)$. By Theorem \ref{dunfordpettis}, this implies that $\FF(\ell_p^n)$ has the Dunford--Pettis property. By Lemma \ref{grothendieck} (ii), we deduce that $\left(\bigoplus\limits_{n=1}^\infty \FF(\ell_p^n) \right )_{\ell_1}$ has the Dunford--Pettis property, and therefore is not isomorphic to $\FF(\ell_p)$.
\end{proof}

\section{Open problems}

\noindent We collect some open problems from the literature on Lipschitz-free spaces over subsets of $\rr^d$, and one problem related to Corollary \ref{l1sum}. Some fundamental questions about Lipschitz-free spaces over $\rr^d$ and its subsets remain unresolved. It is known that $\FF(\rr)$ and $\FF(\rr^2)$ are not isomorphic (see \cite{planarearthmover}), but the corresponding nonisomorphism result for two distinct dimensions, both at least $2$, is unknown.
\begin{problem}
    Are $\FF(\rr^2)$ and $\FF(\rr^3)$ isomorphic? If $m, n\geq 2$ are distinct, are $\FF(\rr^n)$ and $\FF(\rr^m)$ ever isomorphic?
\end{problem}
\noindent The next problem (raised in \cite{Aliaga}*{Question 5}) relates to isomorphism classes of Lipschitz-free spaces over subsets of $\rr^2$, and originally motivated us to consider the Dunford--Pettis property for $\FF(\rr^2)$. It is known that $\FF(\rr\times \mathbb{Z})$ is not isomorphic to any of $\FF(\rr)$, $\FF(\mathbb{Z})$ or $\FF(\mathbb{Z}^2)$. In the first case, the spaces are not isomorphic because their duals are not isomorphic. In the second and third cases, $\FF(\rr\times \mathbb{Z})$ fails the Radon-Nikod\'ym property, while both $\FF(\mathbb{Z})$ and $\FF(\mathbb{Z}^2)$ have the Radon-Nikod\'ym property. It is, however, unknown whether $\FF(\rr\times \mathbb{Z})$ is isomorphic to $\FF(\rr^2)$.
\begin{problem}[\cite{Aliaga}*{Question 5}]
    Are $\FF(\rr\times \mathbb{Z})$ and $\FF(\rr^2)$ isomorphic?
\end{problem}

\noindent For the analogous problem for Lipschitz spaces, Aliaga poses the following problem in \cite{Aliaga}*{Question 3}.

\begin{problem}
    Let $M\subset \mathbb{R}^2$ be infinite. Does it follow that either $\lip(M)\sim \lip(\rr)$ or $\lip(M)\sim \lip(\rr^2)$?
\end{problem}

\noindent The final problem is motivated by the missing $p=1$ case of Corollary \ref{l1sum}. The space $\ell_1$ has the Dunford--Pettis property, so the isometric lifting property does not directly show that $\FF(\ell_1)$ fails the Dunford--Pettis property.

\begin{problem}
    Does $\FF(\ell_1)$ have the Dunford--Pettis property? If so, are $\FF(\ell_1)$ and $\left(\bigoplus\limits_{n=1}^\infty \FF(\ell_1^n) \right)_{\ell_1}$ isomorphic?
\end{problem}

\section*{Acknowledgements}
\noindent The author is supported by a PhD studentship from the Department of Pure Mathematics and Mathematical Statistics at the University of Cambridge, funded by XTX Markets. This work was carried out as part of a PhD under the supervision of Andr\'as Zs\'ak, who the author would like to thank for his guidance and support.


\begin{bibdiv}
\begin{biblist}

\bib{Aliaga}{article}{
   author={Aliaga, Ram\'on J.},
   title={Lipschitz spaces over non-porous sets},
   journal={J. Funct. Anal.},
   volume={290},
   date={2026},
   number={11},
   pages={Paper No. 111439, 17},
   issn={0022-1236},
   review={\MR{5040093}},
   doi={10.1016/j.jfa.2026.111439},
}

\bib{unrectifiable}{article}{
   author={Aliaga, Ram\'on J.},
   author={Gartland, Chris},
   author={Petitjean, Colin},
   author={Proch\'azka, Anton\'in},
   title={Purely 1-unrectifiable metric spaces and locally flat Lipschitz
   functions},
   journal={Trans. Amer. Math. Soc.},
   volume={375},
   date={2022},
   number={5},
   pages={3529--3567},
   issn={0002-9947},
   review={\MR{4402669}},
   doi={10.1090/tran/8591},
}

\bib{normedspaces}{article}{
    author={R. Aliaga},
    author={R. Medina},
    title={Lipschitz extension and Lipschitz-free spaces over nets in normed spaces},
    date={2026},
    status={preprint},
    doi={10.48550/arXiv.2601.03131},
    eprint={https://arxiv.org/pdf/2601.03131v2},
    setup={\let\PrintDatePV\PrintDate}
}

\bib{bourgain}{article}{
   author={Bourgain, J.},
   title={The Dunford--Pettis property for the ball-algebras, the
   polydisc-algebras and the Sobolev spaces},
   journal={Studia Math.},
   volume={77},
   date={1984},
   number={3},
   pages={246--253},
   issn={0039-3223},
   review={\MR{0745280}},
   doi={10.4064/sm-77-3-246-253},
}

\bib{isolipschitzspace}{article}{
   author={Candido, Leandro},
   author={C\'uth, Marek},
   author={Doucha, Michal},
   title={Isomorphisms between spaces of Lipschitz functions},
   journal={J. Funct. Anal.},
   volume={277},
   date={2019},
   number={8},
   pages={2697--2727},
   issn={0022-1236},
   review={\MR{3990732}},
   doi={10.1016/j.jfa.2019.02.003},
}

\bib{quotientrepresentation}{article}{
   author={C\'uth, Marek},
   author={Kalenda, Ond\v rej F. K.},
   author={Kaplick\'y, Petr},
   title={Isometric representation of Lipschitz-free spaces over convex
   domains in finite-dimensional spaces},
   journal={Mathematika},
   volume={63},
   date={2017},
   number={2},
   pages={538--552},
   issn={0025-5793},
   review={\MR{3706595}},
   doi={10.1112/S0025579317000031},
}

\bib{pelczynski}{article}{
   author={Garc\'ia-Lirola, Luis C.},
   author={Proch\'azka, Anton\'in},
   title={Pe\l czy\'nski space is isomorphic to the Lipschitz free space
   over a compact set},
   journal={Proc. Amer. Math. Soc.},
   volume={147},
   date={2019},
   number={7},
   pages={3057--3060},
   issn={0002-9939},
   review={\MR{3973906}},
   doi={10.1090/proc/14446},
}

\bib{godefroykalton}{article}{
   author={Godefroy, G.},
   author={Kalton, N. J.},
   title={Lipschitz-free Banach spaces},
   note={Dedicated to Professor Aleksander Pe\l czy\'nski on the occasion of
   his 70th birthday},
   journal={Studia Math.},
   volume={159},
   date={2003},
   number={1},
   pages={121--141},
   issn={0039-3223},
   review={\MR{2030906}},
   doi={10.4064/sm159-1-6},
}

\bib{R^nquotient}{article}{
   author={Godefroy, Gilles},
   author={Lerner, Nicolas},
   title={Some natural subspaces and quotient spaces of $L^1$},
   journal={Adv. Oper. Theory},
   volume={3},
   date={2018},
   number={1},
   pages={61--74},
   issn={2662-2009},
   review={\MR{3730340}},
   doi={10.22034/aot.1702-1124},
}

\bib{Kaufmann}{article}{
   author={Kaufmann, Pedro Levit},
   title={Products of Lipschitz-free spaces and applications},
   journal={Studia Math.},
   volume={226},
   date={2015},
   number={3},
   pages={213--227},
   issn={0039-3223},
   review={\MR{3356002}},
   doi={10.4064/sm226-3-2},
}

\bib{LeeNaor}{article}{
   author={Lee, James R.},
   author={Naor, Assaf},
   title={Extending Lipschitz functions via random metric partitions},
   journal={Invent. Math.},
   volume={160},
   date={2005},
   number={1},
   pages={59--95},
   issn={0020-9910},
   review={\MR{2129708}},
   doi={10.1007/s00222-004-0400-5},
}

\bib{planarearthmover}{article}{
   author={Naor, Assaf},
   author={Schechtman, Gideon},
   title={Planar earthmover is not in $L_1$},
   journal={SIAM J. Comput.},
   volume={37},
   date={2007},
   number={3},
   pages={804--826},
   issn={0097-5397},
   review={\MR{2341917}},
   doi={10.1137/05064206X},
}

\bib{Toland}{article}{
    author={Toland, John F.},
    title={Localizing weak convergence in $L_\infty$},
    date={2018},
    eprint={arXiv:1802.01878},
}

\end{biblist}
\end{bibdiv}
\end{document}